\documentclass[12pt,reqno]{amsart}
\usepackage[top=2cm,bottom=2cm,right=2.5cm,left=2.5cm]{geometry}
\usepackage{amssymb}
\usepackage{amsmath, amsthm, amscd, amsfonts, amssymb, graphicx, color}
\usepackage[bookmarksnumbered, colorlinks, plainpages]{hyperref}
 
\usepackage{hyperref}

\newtheorem{theorem}{Theorem}[section]
\newtheorem{lemma}[theorem]{Lemma}
\newtheorem{proposition}[theorem]{Proposition}
\newtheorem{corollary}[theorem]{Corollary}
\theoremstyle{definition}
\newtheorem{definition}[theorem]{Definition}
\newtheorem{example}[theorem]{Example}

\theoremstyle{remark}
\newtheorem{remark}[theorem]{Remark}

\numberwithin{equation}{section}

\begin{document}
	
	\setcounter{page}{1}
	
	\title[Modular Biframes for Operators]{Modular Biframes for Operators}
	
	\author[S. E. Oustani, M. Rossafi]{Salah Eddine Oustani$^{1}$ and Mohamed Rossafi$^2$$^{*}$}
	
	\address{$^{1}$Department of Mathematics Faculty of Sciences, University of Ibn Tofail, B.P. 133, Kenitra, Morocco}
	\email{\textcolor[rgb]{0.00,0.00,0.84}{salaheddine.oustani.1975@gmail.com}}
	\address{$^{2}$Department of Mathematics Faculty of Sciences, Dhar El Mahraz University Sidi Mohamed Ben Abdellah, Fes, Morocco}
	\email{\textcolor[rgb]{0.00,0.00,0.84}{rossafimohamed@gmail.com}}
	
	
	
	\date{Received: xxxxxx; Revised: yyyyyy; Accepted: zzzzzz.
		\newline \indent $^{*}$ Corresponding author}

	\begin{abstract}
One of the most important problems in the studying of frames and its extensions is the invariance of these systems under perturbation. The current paper is concerned with  the invariance of Modular biframes for operators  under some class of closed range operators.
	\end{abstract}
	\maketitle
	
	\baselineskip=12.4pt
	
\section{Introduction and Preliminaries  }
	
Frames are basis-like systems that span a vector space but allow for linear dependency, that can be used to  obtain other desirable features unavailable with orthonormal bases.
Theory of frames is a useful tool to expand functions with respect to a system of functions which is, in general, non-orthogonal and overcomplete.
The aim of this theory, developed by Duffin
and Schaeffer \cite{Duff}, was to solve some problems related to the nonharmonic Fourier series. However, the frame theory had not attracted much attention until the celebrated work by Daubechies, Crossman, and Meyer \cite{DG}. Nowaday, Frames attract a steady interests in recent research in applied mathematics because they are used in various areas such as signal processing \cite{FF}, sampling theory \cite{EL}. 
The theory of frames has been rapidly generalized  and, until 2005,  various generalizations consisting of vectors in Hilbert spaces or Hilbert $C^{\ast}$-modules have been developed. 

It is well known that Hilbert $\mathcal{C}^{\ast}$-module is an object like a Hilbert space except that the inner product is not scalar-valued but takes its values in a $\mathcal{C}^{\ast}$-algebra of coefficients. Since the geometry of these modules emerges from the $ \mathcal{C}^{\ast} $-valued inner product, some basic properties of Hilbert spaces like self-duality  must be given up. These modules play an important role in the study of non-commutative geometry, locally compact quantum groups and dynamical systems. For more information, we refer the readers to \cite{Assila, Ghiati, DG, Ch, FL, MANUI, Massit, Gav, Rossafi}.

The main goal of this paper is to study the invariance of $K$-biframes under some closed range operators in Hilbert $\mathcal{C}^{\ast}$-modules.

 In continue,  we aim to review some topics and basic definitions about frames for Hilbert $\mathcal{C}^{\ast}$-module  and operator theory that will be needed later. The reader is referred for instance to \cite{ROssafi, Ch} for more information.
		
\begin{definition}\cite{MANUI}.
	Let $ \mathcal{A} $ be a unital $C^{\ast}$-algebra and $\mathcal{H}$ be a left $ \mathcal{A} $-module, such that the linear structures of $\mathcal{A}$ and $ \mathcal{H} $ are compatible. $\mathcal{H}$ is a pre-Hilbert $\mathcal{A}$-module if $\mathcal{H}$ is equipped with an $\mathcal{A}$-valued inner product $\langle.,.\rangle_{\mathcal{A}} :\mathcal{H}\times\mathcal{H}\rightarrow\mathcal{A}$, such that is sesquilinear, positive definite and respects the module action. In the other words,
	\begin{itemize}
		\item [(i)] $ \langle x,x\rangle_{\mathcal{A}}\geq0 $ for all $ x\in\mathcal{H} $ and $ \langle x,x\rangle_{\mathcal{A}}=0$ if and only if $x=0$.
		\item [(ii)] $\langle ax+y,z\rangle_{\mathcal{A}}=a\langle x,y\rangle_{\mathcal{A}}+\langle y,z\rangle_{\mathcal{A}}$ for all $a\in\mathcal{A}$ and $x,y,z\in\mathcal{H}$.
		\item[(iii)] $ \langle x,y\rangle_{\mathcal{A}}=\langle y,x\rangle_{\mathcal{A}}^{\ast} $ for all $x,y\in\mathcal{H}$.
	\end{itemize}	 
\end{definition}
For $x\in\mathcal{H}, $ we define $||x||=||\langle x,x\rangle||^{\frac{1}{2}}$. If $\mathcal{H}$ is complete with $||.||$, it is called a Hilbert $\mathcal{A}$-module or a Hilbert $C^{\ast}$-module over $\mathcal{A}$. For every $a$ in $C^{\ast}$-algebra $\mathcal{A}$, we have $|a|=(a^{\ast}a)^{\frac{1}{2}}$ and the $\mathcal{A}$-valued norm on $\mathcal{H}$ is defined by $|x|=\langle x, x\rangle^{\frac{1}{2}}$ for $x\in\mathcal{H}$.

Let $\mathcal{H}$ and $\mathcal{K}$ be two Hilbert $\mathcal{A}$-modules, A map $T:\mathcal{H}\rightarrow\mathcal{K}$ is said to be adjointable if there exists a map $T^{\ast}:\mathcal{K}\rightarrow\mathcal{H}$ such that $\langle Tx,y\rangle_{\mathcal{A}}=\langle x,T^{\ast}y\rangle_{\mathcal{A}}$ for all $x\in\mathcal{H}$ and $y\in\mathcal{K}$.

We also reserve the notation $End_{\mathcal{A}}^{\ast}(\mathcal{H},\mathcal{K})$ for the set of all adjointable operators from $\mathcal{H}$ to $\mathcal{K}$ and $End_{\mathcal{A}}^{\ast}(\mathcal{H},\mathcal{H})$ is abbreviated to $End_{\mathcal{A}}^{\ast}(\mathcal{H})$.

\begin{example} 
Let us consider the following set
$$l^{2}\left( \mathcal{A}  \right)=\{\{a_{j}\}_{j \in \mathbb{J}}\subseteq \mathcal{A}: \sum_{j\in \mathbb{J} }a_{j}a_{j}^{\ast} \,\, converge\,\, in\, \mid\mid .\mid\mid_{\mathcal{A}}    \}.$$
It is easy to see that $l^{2}\left( \mathcal{A}  \right)$ with pointwise operations and the inner product
$$   \langle \{a_{j}\},\{b_{j}\}\rangle = \sum_{j\in \mathbb{J}}a_{j}b_{j}^{\ast} ,       $$
is a Hilbert $ \mathcal{C}^{\ast}$-module which is called the standard Hilbert $ \mathcal{C}^{\ast}$-module over $\mathcal{A}$.
\end{example}
A biframe is a pair of sequences in a Hilbert space that
applies to an inequality similar to a frame inequality.
\begin{definition} \cite{Alijani biframe}
A biframe is a pair $\left( \{ x_{j}\}_{j\in \mathbb{J}},  \{ y_{j}\}_{j\in \mathbb{J}}\right)$ of sequences in a Hilbert space such there exist $\alpha, \beta >0$
$$  \alpha \mid \mid x \mid \mid ^{2} \leq   \sum_{j\in \mathbb{J}}\langle x, x_{j}\rangle\langle y_{j},  x\rangle\leq \beta \mid \mid x \mid \mid ^{2}, \, for\, all     \ x \in \mathcal{H}. $$          
\end{definition}

\begin{example} \cite{Alijani biframe}
We consider two following sequences 
$$\{x_{j}\}_{1\leq j \leq 2}=\{(1, 0), (0, 1)\} \,\, and \,\, \{y_{j}\}_{1\leq j \leq 2}=\{(3, 1), (1, 1) \}$$ 
$\left( \{ x_{j}\}_{1\leq j \leq 2  },  \{ y_{j}\}_{1\leq j \leq 2 }\right)$  is a biframe  with bounds $\dfrac{1}{2}$ and 3.			
 \end{example}
Recentely, M. Rossafi introduced the concept of biframes in Hilbert $\mathcal{C}^{\ast}$-modules as a new generalization of modular frames.
	
\begin{definition} \cite{karara  }
Let $\mathcal{H} $ be a Hilbert $\mathcal{A}$-module. A sequence $\left( \{ x_{j}\}_{j\in \mathbb{J}},  \{ y_{j}\}_{j\in \mathbb{J}}\right)$   is said to be a biframe for $\mathcal{H}$, if there exist $\alpha,\beta >0 $ such that
\begin{equation*}
\alpha\langle x, x\rangle \leq \sum_{j\in \mathbb{J}}\langle x, x_{j}\rangle\langle y_{j}, x\rangle\leq \beta\langle x, x\rangle, \, for\, all     \ x \in \mathcal{H}.
\end{equation*}
\end{definition}
For an operator $T\in \mathcal{L}(\mathcal{H}), $ we denote by $R\left( T\right) $ and $N\left( T\right) $ the range and the kernel subspaces of $T$.  We denote by $\mathcal{CR}\left( \mathcal{H}\right) $ the set of all close range operators on $\mathcal{H}$ and $I$ is the identity operator.   As usual, for $E \subset \mathcal{H},$ the orthogonal projection on $E$ is denoted by $\pi_{E}$. We also write $R^{\infty}\left( T\right) = \bigcap_{n\geq 0}R\left( T^{n}\right),$ for the generalized range.
	
The following lemma is a key tool for the proofs of our main results.
		
\begin{lemma} \cite{Pas}\label{lem1.1}
Let $\mathcal{H}$ be Hilbert $\mathcal{A} $-module and $T \in {\mathcal{L}}\left( \mathcal{H}\right). $  Then 
\begin{equation*}
\langle Tx,Tx\rangle \leq \mid\mid T\mid\mid^{2} \langle x, x\rangle, \, for\, all\   x\in \mathcal{H}.
\end{equation*}
			
\end{lemma}
		
\begin{lemma} \cite{decom} \label{lem1.2}
Let $ T,G \in \mathcal{L}\left( \mathcal{H}\right) $
such that $ R(G)$ is closed. Then the following statements are equivalent:
\begin{itemize}
\item[1.] $R(T)\subseteq R(G);$
\item[2.] $\alpha \langle T^{\ast}x, T^{\ast}x\rangle\leq  \langle G^{\ast}x, G^{\ast}x\rangle$, for some $\alpha>0 $.
\end{itemize}
\end{lemma}	
It should be noted that the closeness of range of operators is an attractive problem which appears in operator theory, especially, in the theory of Fredholm operators and generalized inverses. 
\begin{theorem} \cite{Mohammad}\label{th1.1}
Suppose that $ T, G\in \mathcal{CR}\left( \mathcal{H}\right)$ such that $TG=GT$. Then 
$R\left( TG\right) $ is closed.	
\end{theorem}
\begin{definition} \cite{MbeO}
The Reduced minimum modulus of $T$  is defined by	
$$\gamma\left( T\right):=inf\{ \mid \mid Tx\mid \mid,\ x\in \mathcal{H},\  dist\left( x, N\left( T\right)\right) =1                     \}.$$
Formally, we set $ \gamma\left( 0\right) := \infty.$
Clearly $\gamma\left( T\right)> 0$ if and only if $R\left( T\right)$ is closed.
			
\end{definition}
		
\begin{example}
Let $T \in \mathcal{L}\left( \mathbb{C}^{3}\right)$ be defined as follows
$$\begin{array}{ccccc}
T &: & \mathbb{C}^{3} & \longrightarrow & \mathbb{C}^{3} \\
& & \left( x_{1}, x_{2},x_{3}\right)   & \longmapsto & \left( x_{1}, x_{1}, x_{1}\right). \\
\end{array}$$
For $x=\left( x_{1}, x_{2},x_{3}\right) \in \mathbb{C}^{3}$, we have, 
$$\mid \mid Tx \mid \mid=\sqrt{3}\mid x_{1}\mid,$$ and 
$$dist\left( x, N\left( T\right)\right)=\mid x_{1}\mid. $$
Then $$ \gamma\left( T\right)=  \sqrt{3}.$$ 
\end{example}
In addition, the concept of semi-regularity  has 
benefited from the work of many authors, in particular from the work of Mbekhta \cite{MbeO} and Rakocevi$\grave{c}$ \cite{Ra}. 
\begin{definition} \cite{MbeO}
An operator $T \in \mathcal{L}\left( \mathcal{H}\right)$ is said to be semi-regular if $R(T)$ is closed and $N\left( T\right) \subset R\left( T^{n}\right),$ for every $n\geq 1 $.
		
\end{definition}
\begin{example}
All surjective and all injective operators with closed range are semi-regular. Some examples of semi-regular operators may be found in \cite{Lab}.
\end{example}
Next, we collect below some  useful properties related to  semi-regular operators
\begin{proposition} \cite{AF} \label{prop1.2}
Let $T$ be semi-regular. Then
\begin{itemize}
\item [1]. $R^{\infty}\left(T\right)$ is closed;
\item [2]. $R^{\infty}\left(T\right)=T\left( R^{\infty}\left(T\right)\right); $
\item [3]. $T-\lambda I$ is semi-regular, for all $ \mid\lambda \mid < \gamma\left( T\right).$
\end{itemize}
\end{proposition}
Recall that the semi-regular resolvent of a bounded operator $T$ is defined by
	$$ reg\left( T\right) =\{    \lambda \in\mathbb{C}: T-\lambda I \,\, is \,\, semi\,-\,regular \}.$$
	
\begin{theorem}\cite{AF} \label{th1.2}
Let $T$ be semi-regular and $\Omega $ be a connected component of $reg\left( T\right) $ and $\lambda_{0} \in \Omega$, then $  R^{\infty}\left(T-\lambda I\right)=  R^{\infty}\left(T-\lambda_{0} I\right)$, for every $ \lambda \in \Omega.$
\end{theorem}
	
In \cite{XQS}, Recall that the Moore-Penrose inverse of an operator $T \in\mathcal{L}\left( \mathcal{H}\right)$ with closed range is defined as the unique operator $T^{\dagger} \in \mathcal{L}\left( \mathcal{H}\right) $ such that:
$$  TT^{\dagger}x=x, \,\, for \,\, every\,\, x \in R\left( T\right).$$
			
\begin{example} \cite{salah}
Let $T \in \mathcal{CR}\left( \mathcal{H}\right) $  such that $T^{2}=T$. We have
$$T^{\dagger}=T^{\dagger}TT^{\dagger}= T^{\dagger}TTT^{\dagger}=P_{R\left( T^{\ast}\right) }P_{R\left( T\right) }. $$	
\end{example}	
Now, we list below some useful properties related to Moore-Penrose inverses. 
\begin{proposition} \cite{SNG} \label{prop1.2}
Let $T \in\mathcal{L}\left( \mathcal{H}\right) $ be a closed range. Then 
\begin{itemize}
\item [1.] $R\left( T^{\dagger}\right) = R\left( T^{\ast}\right)=N\left( T\right) ^{\perp};$ 
\item [2.] $N\left(  T^{\dagger}\right) = N\left(  T^{\ast}\right)=R\left( T\right) ^{\perp};$
\item [3.]  $T^{\dagger \ast}=T^{\ast \dagger}.$	
\end{itemize}
\end{proposition}
The reader is referred for instance to \cite{israel, HARTE} for more information.

\begin{theorem} \cite{Mohammad} \label{th1.3}	Let $ T \in  \mathcal{CR}\left( \mathcal{H}\right) $ be closed range and  $ G \in  \mathcal{L}\left( \mathcal{H}\right)  $ be an arbitrary operator which commutes with $ T.$ Then $G$ commutes with $T^{\dagger}$.
\end{theorem}

EP matrix, as an extension of normal matrix, has been extended by Campbell and Meyer \cite{ep} to operators with closed range on a Hilbert space.
	
\begin{definition} \cite{Sharifi}
An operator $ T \in \mathcal{L}\left( \mathcal{H}\right) $ is called an EP operator if $ R\left( T\right) $ is closed and $R\left( T\right) =R\left( T^{\ast}\right) .$  
\end{definition}
\begin{example}
Let $T\in \mathcal{L}\left( l^{2}\left( \mathbb{C}\right)\right)  $ be defined as follows:
$$ T\left( \left( x_{j}\right) _{j \geq 1}\right) = \left(\left( y_{j}\right) _{j \geq 1} \right),$$ where
		$$
		y_{j}= \left\{
		\begin{array}{ll}
			x_{1}-x_{3} & \mbox{if } j=1 \\
			0 & \mbox{if } j=2 \\
			x_{j}  &\mbox{if $j\geq 3$}\\
		\end{array}
		\right.
		$$
		By some straightforward computations, we obtain that $T$ is an EP-operator.
\end{example}
	
\begin{proposition} \cite{MMM} \label{prop1.4}
Let $ T\in \mathcal{L}\left( \mathcal{H}\right)  $ be a normal operator with closed range. Then $T$ is an EP operator.
\end{proposition}

\section{Main results }
In this section, we begin with the following definition.  
\begin{definition}
 A pair $\left( \{x_{j}\}_{j\in \mathbb{J}}, \{y_{j}\}_{j\in \mathbb{J}}\right)   $ of sequences in $\mathcal{H}$ is called $K$-biframes for $ \mathcal{H}$, if there exist $\alpha,\beta >0 $ such that 
\begin{equation*}
\alpha  \langle K^{\ast}x, K^{\ast}x\rangle \leq\sum_{j\in\mathbb{J}}\langle x,x_{j}\rangle\langle y_{j},x\rangle   \leq \beta  \langle x, x\rangle,\,\,\left( \forall x \in \mathcal{H}\right).
\end{equation*}		
$\alpha$ and $\beta$  are called lower and upper $K$-biframe bounds, respectively.
\end{definition}

To throw more light on the subject and understand the use of this concept, we exhibit below some examples of the $K$-biframes.
\begin{example}
 Let $K \in \mathcal{L}\left(\mathbb{C}^{3} \right) $ be defined by
	  $$
		Ke_{j}= \left\{
		\begin{array}{ll}
			3e_{1}& \mbox{if } j=1 \\
			e_{j} & \mbox{if } j=2,3 
		\end{array}
		\right.
		$$
		where $\{e_{j}\}_{j\geq 1}$ is an orthonormal basis for $\mathbb{C}^{3}. $
		Obviously, we have
		$$
		K^{\ast}e_{j}= \left\{
		\begin{array}{ll}
			3e_{1}& \mbox{if } j=1 \\
			e_{j} & \mbox{if } j=2,3 
		\end{array}
		\right.
		$$
		For $x=\left( x_{1},x_{2},x_{3}\right) \in \mathbb{C}^{3} $, we get
		$$ \mid\mid K^{\ast}x\mid\mid^{2}=9\mid x_{1}\mid ^{2}+ \mid x_{2}\mid ^{2}+\mid x_{3}\mid ^{2}.$$
		Now, consider, for $z \in \mathbb{R}$
		$$
		\theta_{j}= \left\{
		\begin{array}{ll}
			e^{iz} e_{1} & \mbox{if } j=1 \\
			\dfrac{1}{2}e_{2} & \mbox{if } j=2 \\
			e_{3}  &\mbox{if $j=3$}\\
		\end{array}
		\right.
		and \,\,
		\psi_{j}= \left\{
		\begin{array}{ll}
			e^{iz}e_{1} & \mbox{if } j=1 \\
			e_{2} & \mbox{if } j=2\\
			\dfrac{1}{3}e_{3} & \mbox{if } j=3
		\end{array}
		\right.
		$$
		hence
		$$ \sum_{j=1}^{3} \langle x,\theta_{j}\rangle \langle\psi_{j},x\rangle  =  \mid x_{1}\mid ^{2}+ \dfrac{1}{2}\mid x_{2}\mid ^{2}+\dfrac{1}{3}\mid x_{3}\mid ^{2}.                                              $$
		Thus
		$$  \dfrac{1}{9}\mid\mid K^{\ast}x \mid\mid ^{2} \leq\sum_{j=1}^{3} \langle x,\theta_{j}\rangle \langle\psi_{j},x\rangle \leq \mid\mid x\mid \mid^{2}.                    $$
		Which implies that $\left( \{\theta_{j}\}_{1\leq j \leq3}, \{\psi_{j}\}_{1\leq j \leq3}\right) $ is a $K$-biframe for $\mathbb{C}^{3}.$
	
\end{example}
	
\begin{example}
Let  $\{e_{n} \}_{n \geq 1}$ be an orthonormal basis of $l^{2}\left(  \mathbb{C}\right)$ and $K \in \mathcal{L}\left(l^{2}\left(  \mathbb{C}\right) \right)$ be defined as follows
$$\begin{array}{ccccc}
K &: & l^{2}\left(  \mathbb{C}\right)  & \longrightarrow & l^{2}\left(  \mathbb{C}\right)  \\
& & \left( x_{1}, x_{2},...\right)   & \longmapsto & \left( 0, x_{1}, x_{2}, x_{3},...\right) . \\
\end{array}$$
Clearly, we have
$$\begin{array}{ccccc}
K^{\ast} &: & l^{2}\left(  \mathbb{C}\right)  & \longrightarrow & l^{2}\left(  \mathbb{C}\right)  \\
& & \left( x_{1}, x_{2},...\right)   & \longmapsto & \left( x_{2}, x_{3},...\right) . \\
\end{array}$$
By setting, for $a \in \mathbb{R}$
$$f_{1}=e ^{ia} e_{1} \,\,\, and \,\,\, f_{j}=e_{j}, \,\,for \,\, j \geq 2.$$
and
$$  g_{1}=\dfrac{1}{\sqrt{2}}e ^{ia} e_{1} \,\,\, and \,\,\, g_{j}=e_{j}, \,\,for \,\, j \geq 2.$$ 
For $x=\left( x_{j}\right)_{j \geq 1} \in l^{2}\left( \mathbb{C}\right) $, we obtain
$$    \sum_{j\geq 1}\langle x,f_{j}\rangle \langle g_{j}, x\rangle= \dfrac{1}{\sqrt{2}}\mid x_{1} \mid ^{2} +\sum_{j \geq 2 }\mid x_{j} \mid ^{2},                  $$
	and
$$\mid\mid K^{\ast}x \mid\mid ^{2}=\sum_{j \geq 1 }\mid x_{j} \mid ^{2}.$$
Thus
$$ \dfrac{1}{\sqrt{2}}\mid\mid K^{\ast}x \mid\mid ^{2}\leq\sum_{j\geq 1}\langle x,f_{j}\rangle \langle g_{j}, x\rangle\leq \mid\mid x\mid\mid ^{2}.                           $$
Consequently, $ \left( \{f_{j}\}_{j \geq 1}, \{g_{j}\}_{j \geq 1} \right)  $ is a $K $-biframe for $ l^{2}\left(  \mathbb{C}\right). $ 
\end{example}
\begin{remark}
Two $K$-frames may not form a $K$-biframe.\\
Indeed, define
$K \in \mathcal{L}\left(\mathbb{C}^{r} \right) $ be defined as follows $$Ke_{1} = \sqrt{2}e_{1}\,\,and \,\, Ke_{j} = e_{j}, \,for \,\, 1\leq j \leq r.$$
where $\{e_{j}\}_{1\leq j \leq r}$ is an orthonormal basis for $\mathbb{C}^{r}. $ Thus
$$K^{\ast}e_{1}=\sqrt{2}e_{1}\,\, and\,\, K^{\ast}e_{j}=e_{j}, \,\, for \,\,1\leq j \leq r.$$
For $x=\left( x_{j}\right)_{1\leq j \leq r} \in \mathbb{C}^{r} $, we get
$$ \mid\mid K^{\ast}x\mid\mid^{2}=2\mid x_{1}\mid ^{2}+ \sum_{j=2}^{r}\mid x_{j}\mid ^{2}.$$
Now, Consider the sequences
$$\{\theta_{j}\}_{1\leq j \leq r} =\{e_{1}, e_{1}, e_{2},e_{2},..,e_{r}      \},\,\,and \,\, \{\psi_{j}\}_{1\leq j \leq r} =\{\dfrac{1}{2}e_{1}, -\dfrac{1}{2}e_{1}, e_{2},e_{2},..,e_{r}      \}. $$
Direct computations show that $\{\theta_{j}\}_{1\leq j \leq r} $ is a $K$-frame with bounds 1 and 2 and  $ \{\psi_{j}\}_{1\leq j \leq r} $
is a $K$-frame with bounds  $ \dfrac{1}{6}$ and 1.\\
On the other hand,  we have
$$ \sum_{j=1}^{r} \langle x,\theta_{j}\rangle \langle\psi_{j},x\rangle = \sum_{j=2}^{r} \mid\langle x,e_{j}\rangle \mid ^{2},$$
hence, if we set $x=e_{1}$, we obtain
$$ \sum_{j=1}^{r} \langle x,\theta_{j}\rangle \langle\psi_{j},x\rangle = 0  $$
Therefore,  $\left( \{\theta_{j}\}_{1\leq j \leq r}, \{\psi_{j}\}_{1\leq j \leq r}\right) $ is not a $K$-biframe for $\mathbb{C}^{r}.$
\end{remark}
The next theorem presents some operators that
preserve the $K$-biframe property of a given $K$-biframe.
\begin{theorem}
Let $L \in \mathcal{L}\left( \mathcal{H}\right) $ such that $KLK=K$ and $\left( \{x_{j}\}_{j\in \mathbb{J}}, \{y_{j}\}_{j\in \mathbb{J}}\right) $ be a $K$-biframe for $\mathcal{H}.$ Then $\left( \{\left( KL\right) x_{j}\}_{j\in \mathbb{J}}, \{\left( KL\right) y_{j}\}_{j\in \mathbb{J}}\right) $ is a $K$-biframe for $ \mathcal{H}$.
\end{theorem}
	
\begin{proof}
First, there exist $ \alpha, \beta >0  $ such that, for all $x\in \mathcal{H}$
\begin{equation*}
\alpha\langle K^{\ast}x,K^{\ast} x\rangle \leq   \sum_{j\in\mathbb{J}}\langle x, x_{j}\rangle\langle y_{j},  x\rangle\leq \beta\langle x, x\rangle,  
\end{equation*}
Since
	$$KLK=K.$$
Thus
		$$ 	\alpha \langle \left( KLK\right)^{\ast}x, \left( KLK\right)^{\ast}x \rangle   \leq\sum_{j\in\mathbb{J}}  \langle x, \left( KL\right) x_{j}\rangle\langle \left( KL\right) y_{j}, x\rangle  \leq \beta \langle  \left( KL\right)^{\ast} x, \left( KL\right)^{\ast}x\rangle.$$
		By Lemma \ref{lem1.1}, we obtain
\begin{equation*}
			\alpha \langle K^{\ast}x,K^{\ast} x\rangle \leq   \sum_{j\in\mathbb{J}}\langle  x, \left( KL\right) x_{j}\rangle\langle\left(  KL\right)  y_{j}, x\rangle \leq \beta \mid\mid KL\mid\mid ^{2}\langle x, x\rangle.
\end{equation*}
Then,  $\left( \{\left( KL\right) x_{j}\}_{j\in \mathbb{J}}, \{\left( KL\right) y_{j}\}_{j\in \mathbb{J}}\right) $ is a $K$-biframe for $ \mathcal{H}$.
\end{proof}
	
\begin{example}
Let $K, L \in \mathcal{L}\left( \mathbb{C}^{2}\right)$ be defined as follows:
\begin{equation*}
			K=
			\begin{pmatrix}
				e^{ia} & e^{-ia} \\	
				0 &   0

			\end{pmatrix}
			and \,\,L=
			\begin{pmatrix}
				e^{-ia}  & 0\\	
				0 &   0.	
			\end{pmatrix}, \,\,
			where \,\, a \in \mathbb{R} \,\, and\,\, i^{2}=-1 .
\end{equation*}	
		For $x=\left( x_{1}, x_{2}\right) \in \mathbb{C}^{2}$, we have
		$$\mid\mid K^{\ast}x \mid\mid ^{2}= 2\mid x_{1}\mid ^{2}.$$
		By setting
		$$
		f_{j}= \left\{
		\begin{array}{ll}
			e^{ia} e_{1} & \mbox{if } j=1 \\
			e_{2} & \mbox{if } j=2 \\
		\end{array}
		\right.
		and \,\, 
		g_{j}= \left\{
		\begin{array}{ll}
			e^{ia} e_{1} & \mbox{if } j=1 \\
			2e_{2} & \mbox{if } j=2 \\
		\end{array}
		\right.$$
		We obtain
		$$ \sum_{j=1}^{2} \langle x,f_{j}\rangle\langle g_{j}, x\rangle = \mid x_{1}\mid^{2} + 2\mid x_{2}\mid^{2}.$$
		Thus
		$$  \dfrac{1}{2}\mid\mid K^{\ast}x \mid\mid ^{2} \leq\sum_{j=1}^{j=2} \langle x,f_{j}\rangle\langle g_{j}, x\rangle \leq 2\mid\mid x\mid \mid^{2}.                              $$
		Then, $\left( \{f_{1}, g_{1}\}, \{f_{2}, g_{2}\}\right) $ is a $K$-biframe for $\mathbb{C}^{2}.$\\
		By some straightforward computations, we obtain 
		$$ KL= 	
		\begin{pmatrix}
			1 & 0 \\	
			0 &   0	
		\end{pmatrix}. $$
		Thus
		$$  \sum_{j=1}^{2} \langle x,\left( KL\right) f_{j}\rangle \langle \left( KL
		\right) g_{j}, x\rangle=  \mid x_{1}\mid^{2}.                            $$
		Then
		$$\dfrac{1}{2}\mid\mid K^{\ast}x \mid\mid ^{2} \leq\sum_{j=1}^{2} \langle x,\left( KL\right)  f_{j}\rangle\langle \left( KL\right)  g_{j}, x\rangle  \leq \mid\mid x\mid \mid^{2}.  $$
		Consequently, $\left( \{\left( KL\right) x_{j}\}_{j\in \mathbb{J}}, \{\left( KL\right) y_{j}\}_{j\in \mathbb{J}}\right) $ is a $K$-biframe for $ \mathbb{C}^{2}.$
\end{example}
In what follows, we assume that $T$ is semi-regular such that $KT=TK$ and $\Omega $ is a connected component of $reg\left( T\right)$ and we agree to use the following notation
$$ T_{\lambda}=T-\lambda I \,\,and\,\, \mathcal{H}_{0} =R^{\infty}\left(T-\lambda_{0}I \right),\,\,\left(\lambda,  \lambda_{0} \in \Omega\right). $$
		
\begin{theorem}
Let $\left( \{x_{j}\}_{j\in \mathbb{J}},\{y_{j}\}_{j\in \mathbb{J}}\right)  $ be a $K$-biframe for $ \mathcal{H}.$  Then $\left( \{ T_{\lambda} x_{j}\}_{j\in \mathbb{J}},\{ T_{\lambda} y_{j}\}_{j\in \mathbb{J}}\right) $ is a $K$-biframe for $  \mathcal{H}_{0}, $ for every $ \lambda \in \Omega.$
\end{theorem}

\begin{proof}
Let  $ x \in \mathcal{R}_{0}.$ According to Theorem \ref{th1.2}, there exists $y \in \mathcal{R}_{0}$ such that
$$   x =T_{\lambda}\left( y\right),    $$
hence
$$  K\left( x\right)=\left( KT_{\lambda}\right)\left( y\right)=  \left( T_{\lambda}K\right)\left( y\right). $$
It follows from Theorem \ref{th1.1}, that $R\left( T_{\lambda}K\right) $ is closed.\\	By Lemma \ref{lem1.2}, there is exists $\xi>0$ such that
$$  \xi \langle K^{\ast}x,K^{\ast} x\rangle  \leq  \langle \left( T_{\lambda}K\right) ^{\ast}x,\left( T_{\lambda}K\right) ^{\ast}x\rangle.           $$
Since $\left( \{x_{j}\}_{j\in \mathbb{J}}, \{y_{j}\}_{j\in \mathbb{J}}\right)  $ is a  $K$-biframe  for $  \mathcal{H}$ with biframe bounds $ \alpha, \beta.$ Then
$$   \alpha \langle \left( T_{\lambda}K\right) ^{\ast}x,\left( T_{\lambda}K\right) ^{\ast}x\rangle  \leq\sum_{j\in \mathbb{J}}\langle  x, T_{\lambda} x_{j}\rangle\langle T_{\lambda} y_{j}, x\rangle\leq \beta  \langle T_{\lambda}^{\ast}x, T_{\lambda}^{\ast}x\rangle. $$
	Using Lemma \ref{lem1.2}, we obtain
$$     \langle T_{\lambda}^{\ast}x, T_{\lambda}^{\ast}x\rangle \leq \mid \mid T_{\lambda} \mid \mid ^{2} \langle x,x\rangle.                              $$
Therefore
$$ \alpha\xi \langle K^{\ast}x,K^{\ast}x\rangle \leq\sum_{j\in \mathbb{J}}\langle  x, T_{\lambda} x_{j}\rangle\langle T_{\lambda} y_{j}, x\rangle\leq \beta \mid \mid T_{\lambda} \mid \mid ^{2} \langle x, x\rangle.$$
This completes the proof.	
\end{proof}

In what follows, Consider  $T\in \mathcal{CR}\left( \mathcal{H}\right)$ and  we fix the following notation:
	$$\varphi\left( K\right) = K^{\dagger \ast}.$$

\begin{lemma} \label{lem2.1}
	Let $K \in\mathcal{CR}\left( \mathcal{H}\right) $. Then
	$$R\left(\varphi \left( K\right)  \right) =R\left( K\right). $$	
				
\end{lemma}
				
\begin{proof}
By Proposition \ref{prop1.2}, we have
$$R\left( \varphi\left( K\right) \right) = R\left( \left( K^{\dagger }\right) ^{\ast}\right)= N\left( K^{\dagger}\right) ^{\perp}= \left( N\left( K^{\ast}\right)\right) ^{\perp}= R\left( K\right).$$
The result is obtained.
	\end{proof}
				
Next, we aimed to describe some invariant subset of $ \mathcal{L}\left( \mathcal{H}\right)  $

\begin{theorem} \label{th2.3}
The following subset
$$ \Gamma=\{ K \in \mathcal{CR}\left( \mathcal{H}\right)  :\left( \{ x_{j} \}_{j\in \mathbb{J}}, \{ y_{j} \}_{j\in \mathbb{J}}\right)\,\, is\,\, a\,\, K-biframe \,\, for \,\, \mathcal{H}  \} $$
is $\varphi$-invariant.
Moreover, we have
$ \pi_{R\left( K\right) }\in \Gamma  $
\end{theorem}

\begin{proof}
Let $K \in \Gamma$, there exist $\alpha, \beta >0$ such that
\begin{equation*}
\alpha \langle K^{\ast}x,K ^{\ast}x\rangle \leq\sum_{j\in\mathbb{J}}\langle x,x_{j}\rangle \langle y_{j},x\rangle \leq \beta\langle x, x\rangle, \,\, \left( x \in  \mathcal{H}\right) .  
\end{equation*}
By Lemma \ref{lem1.2} and Lemma \ref{lem1.1}, there exists $\xi >0$ such that
$$ \xi \langle \varphi\left( K\right) ^{\ast}x,\varphi\left( K\right) ^{\ast}x\rangle \leq \langle K^{\ast}x,K ^{\ast}x\rangle $$
This implies
\begin{equation*}
\alpha \xi \langle \varphi\left( K\right) ^{\ast}x,\varphi\left( K\right)  ^{\ast}x\rangle \leq\sum_{j\in\mathbb{J}}\langle x,x_{j}\rangle \langle y_{j},x\rangle \leq \beta\langle x, x\rangle.
\end{equation*}
Therefore $$\varphi\left( \Gamma\right) \subset \Gamma $$\\
 by Lemma \ref{lem1.1}, we have, For $x \in \mathcal{H}$
$$ \langle \left(\varphi\left( K\right)K^{\ast} \right) ^{\ast}x ,\left(\varphi\left( K\right)K^{\ast} \right) ^{\ast}x\rangle \leq \mid\mid K \mid\mid^{2} \langle \left(\varphi\left( K\right) \right) ^{\ast}x ,\left(\varphi\left( K\right) \right) ^{\ast}x\rangle. $$
Since
$$  \varphi\left( K\right)K^{\ast} = \left( KK^{\dagger}\right) ^{\ast}=\left( \pi_{R\left( K\right) }\right)^{\ast}= \pi_{R\left( K\right) }.$$
Thus
\begin{equation*}
\alpha  \mid\mid K \mid\mid^{-2}\langle P_{R\left( K\right) }^{\ast}x,P_{R\left( K\right) }^{\ast}x\rangle  \leq\sum_{j\in\mathbb{J}}\langle x,x_{j}\rangle \langle y_{j},x\rangle \leq \beta \langle x,x\rangle.
\end{equation*}	
	Therefore  $ \pi_{R\left( K\right) } \in \Gamma $
\end{proof}
	
\begin{proposition}
Let $\left( \{x_{j}\}_{j\in \mathbb{J}},\{y_{j}\}_{j\in \mathbb{J}}\right) $ be a $K$-biframe for $\mathcal{H}$  and $U \in \mathcal{L}\left( \mathcal{H}\right) $ be  unitary. Then
$\left( \{Ux_{j}\}_{j\in \mathbb{J}}, \{Uy_{j}\}_{j\in \mathbb{J}} \right) $ is a $\left( U\varphi\left( K\right) U^{\ast}\right) $-biframe for $\mathcal{H}.$
\end{proposition}
\begin{proof}
Direct computations show that $$\left( UK^{\ast}U^{\ast}\right) ^{\dagger}=UK^{\ast\dagger}U^{\ast}. $$ This implies, 
$$\varphi\left( UKU^{\ast}\right) = \left( UKU^{\ast}\right) ^{\ast\dagger}= U\varphi\left( K\right) U^{\ast}.$$ 
For $x\in \mathcal{H},$ we have 
$$\langle \left( UKU^{\ast}\right) ^{\ast}x,\left( UKU^{\ast}\right) ^{\ast}x\rangle \leq \langle \left( K^{\ast}U^{\ast}\right) x,\left( K^{\ast}U^{\ast}\right)x\rangle $$
Thus, there exist $\alpha, \beta>0 $ such that
$$ \alpha \langle \left( UKU^{\ast}\right) ^{\ast}x,\left( UKU^{\ast}\right) ^{\ast}x\rangle \leq \sum_{j \in \mathbb{J}}\langle x,Ux_{j}\rangle \langle \left( Uy_{j}\right) , x\rangle   \leq \beta \langle x, x\rangle.  $$
Then, $\left( \{Ux_{j}\}_{j\in \mathbb{J}}, \{Uy_{j}\}_{j\in \mathbb{J}} \right) $ is a  $\left( UKU^{\ast}\right)$-frame for $\mathcal{H}.$\\
		Since
$$\varphi\left(  UKU^{\ast}\right) =U\varphi\left( K\right) U^{\ast}.$$
By Theorem \ref{th2.3}, $\left( \{Ux_{j}\}_{j\in \mathbb{J}}, \{Uy_{j}\}_{j\in \mathbb{J}} \right) $ is a $\left( U\varphi\left( K\right) U^{\ast}\right) $-biframe for $\mathcal{H}.$
\end{proof}

\begin{proposition} \label{prop2.2}
Let $T\in \mathcal{L}\left( \mathcal{H}\right) $ be EP. Then
		 $\varphi\left( T\right) $ is EP too.
\end{proposition}
	
\begin{proof}
By Lemma \ref{lem2.1}, we obtain
$$  R\left( \varphi\left( T\right)^{\ast} \right) =R\left( \varphi\left( T^{\ast}\right) \right)=R\left( T^{\ast}\right)$$
Since $T$ is EP, we have
$$  R\left( \varphi\left( T\right)^{\ast} \right)=  R\left( \varphi\left( T\right) \right)  $$
Therefore, $\varphi\left( T\right) $ is EP.
	\end{proof}

  In the next, for given an appropriate operator $T$, we intend to construct some $K$-biframes for $R (T ).$
\begin{theorem} \label{theo2.4}
Let $T$ be EP such that $ T^{\ast}K=KT^{\ast} $ and $\left( \{ x_{j} \}_{j\in \mathbb{J}}, \{ y_{j} \}_{j\in \mathbb{J}}\right)$ be a $K$-biframe for $\mathcal{H}.$ Then  $\left( \{\varphi\left( T\right) x_{j} \}_{j\in \mathbb{J}}, \{\varphi\left( T\right) y_{j} \}_{j\in \mathbb{J}}\right) $ is a $K$-biframe for $ R\left( T\right).$	
\end{theorem}
\begin{proof}
Let $x\in R\left( \varphi\left( T\right)\right) ,$ we have 
\begin{equation*}
K\left( x\right)=K\left(\varphi\left( T\right)\varphi\left( T\right) ^{\dagger} x\right).
\end{equation*} 
		By Theorem \ref{th1.3}, we get
		$$  \varphi\left( T\right)K=K\varphi\left( T\right).  $$
		Hence
		$$K\left( x\right)=\left(\varphi\left( T\right)K \right)\left( \varphi\left( T\right) ^{\dagger} x\right).$$
		It follows from Theorem \ref{th1.1}, that  $R\left(  \varphi\left( T\right)K\right) $ is closed.\\   
		Using Proposition \ref{prop2.2}, we obtain
			$$\varphi\left( T\right)^{\dagger}x \in R\left( \varphi \left( T\right)^{\ast}\right) =R\left( T\right) .$$                         
		By Lemma \ref{lem1.2}, there exist $\xi>0$ such that
		\begin{equation*}
			\xi\langle K^{\ast} x,K^{\ast}x\rangle\leq \langle \left( \varphi\left( T\right)K\right) ^{\ast} x,\left( \varphi\left( T\right)K\right) ^{\ast}x\rangle
		\end{equation*}
		and by Lemma \ref{lem1.1}, we have
		$$  \langle \varphi \left( T\right)^{\ast}x, \varphi \left( T\right)^{\ast}x\rangle \leq \mid\mid T^\dagger\mid\mid ^{2}\langle x, x\rangle .            $$
		Since, $\left( \{x_{j}\}_{j\in \mathbb{J}}, \{y_{j}\}_{j\in \mathbb{J}}\right)  $ is a $K$-biframes for $\mathcal{H},$ there exist $\alpha, \beta > 0$ such that
		$$\alpha  \langle K^{\ast}x,K^{\ast} x\rangle \leq\sum_{j\in\mathbb{J}}\langle x,x_{j}\rangle\langle y_{j},x\rangle   \leq \beta  \langle x, x\rangle.$$
		Thus
		$$\alpha \langle \left( \varphi \left( T\right)K\right) ^{\ast}x, \left( \varphi \left( T\right)K\right) ^{\ast}x\rangle \leq \sum_{j\in\mathbb{J}}\langle x,\varphi\left( T\right)x_{j}\rangle \langle \varphi\left( T\right)y_{j}, x\rangle  \leq \beta  \langle \varphi \left( T\right)^{\ast}x, \varphi \left( T\right)^{\ast}x\rangle  . $$
		By Lemma \ref{lem1.1}, we get
		$$        \langle \varphi \left( T\right)^{\ast}x, \varphi \left( T\right)^{\ast}x\rangle \leq \mid\mid T^\dagger\mid\mid ^{2}  \langle x, x\rangle      $$
		Therefore
		$$  \alpha\xi \langle K^{\ast} x,K^{\ast}x\rangle \leq \sum_{j\in\mathbb{J}}\langle x,\varphi\left( T\right) x_{j}\rangle \langle  \varphi\left( T\right) y_{j} ,x\rangle  \leq \beta \mid\mid T^\dagger\mid\mid ^{2}  \langle x, x\rangle.$$
		This completes the proof. 
	\end{proof}
	
	By assumption of Theorem \ref{theo2.4}, we obtain that
\begin{corollary}
Let $T$ be a normal operator. Then   $\left( \{ \varphi\left( T\right)x_{j} \}_{j\in \mathbb{J}}, \{ \varphi\left( T\right)y_{j} \}_{j\in \mathbb{J}}\right)$ is a $K$-biframe for $R\left( T\right).$   
\end{corollary}

\begin{proof}
It follows from  Proposition \ref{prop1.4}.
\end{proof}

\medskip

\section*{Declarations}

\medskip

\noindent \textbf{Availablity of data and materials}\newline
\noindent Not applicable.

\medskip

\noindent \textbf{Competing  interest}\newline
\noindent The authors declare that they have no competing interests.

\medskip

\noindent \textbf{Fundings}\newline
\noindent  Authors declare that there is no funding available for this article.

\medskip

\noindent \textbf{Authors' contributions}\newline
\noindent The authors equally conceived of the study, participated in its
design and coordination, drafted the manuscript, participated in the
sequence alignment, and read and approved the final manuscript.

\end{document}